\documentclass[a4paper]{article}
\usepackage{amsfonts}
\usepackage{amssymb,amsthm, amsmath,graphicx,bm,ebezier,amsthm}
\usepackage{hyperref}
\usepackage{url,lineno}
\usepackage{arydshln}
\usepackage[ruled,vlined,linesnumbered]{algorithm2e}
\usepackage[dvipsnames]{xcolor}
\usepackage{float}

\usepackage{tikz}
\usetikzlibrary{arrows.meta}
\tikzset{
  my tree/.style={
    <-, 
    nodes={minimum size = .25cm},
    >=Stealth[],
   level 1/.style={sibling distance=30pt}
  },
}

\newtheorem{theorem}{Theorem}
\newtheorem{definition}[theorem]{Definition}
\newtheorem{proposition}[theorem]{Proposition}
\newtheorem{corollary}[theorem]{Corollary}
\newtheorem{lemma}[theorem]{Lemma}

\theoremstyle{remark}

\newtheorem{example}[theorem]{Example}

\usepackage{todonotes}

\def\CaH{\mathcal{H}}

\def\CaC{\mathcal{C}}

\def\N{\mathbb{N}}
\def\Z{\mathbb{Z}}
\def\Q{\mathbb{Q}}

\def\Ap{\mathrm{Ap}}
\def\Fb{\mathrm{Fb}}
\def\mult{\mathrm{mult}}

\title{On ideals of affine semigroups and affine semigroups with maximal embedding dimension}

\date{}
\author{
J. I. Garc\'{\i}a-Garc\'{\i}a,
R. Tapia-Ramos,
and A. Vigneron-Tenorio
}

\begin{document}

\maketitle

\begin{abstract}
Let $S\subseteq \N^p$ be a semigroup, any $P\subseteq S$ is an ideal of $S$ if $P+S\subseteq P$, and an $I(S)$-semigroup is the affine semigroup $P\cup \{0\}$, with $P$ an ideal of $S$. We characterise the $I(S)$-semigroups and the ones that also are $\CaC$-semigroups. Moreover, some algorithms are provided to compute all the $I(S)$-semigroups satisfying some properties. From a family of ideals of $S$, we introduce the affine semigroups with maximal embedding dimension, characterising them and describing some families.
\end{abstract}

{\small



\section*{Introduction}
An affine semigroup $S$ of $\N^p$ (for a non-zero natural number $p$) is a commutative additive submonoid of $\N^p$ containing the zero element, and such a finite generating set exists. That is, there exists a finite subset $\{n_1,\ldots ,n_r\}\subset S$ satisfying $S= \left\{ \sum _{i=1}^r \lambda_in_i\mid \lambda_1,\ldots ,\lambda_r\in \N\right\}$. If $p=1$ and generators $n_1,\ldots ,n_r$ are coprime, the semigroup $S$ is called numerical semigroup. Consider the non-negative integer cone generated by a set $B\subset \N^p$ as the set 
$$\CaC_B=\left\{ \sum _{i=1}^k \lambda_i b_i\mid k\in \N,\, \lambda_1,\ldots ,\lambda_k\in \Q_{\ge 0},\mbox{ and } b_1,\ldots ,b_k\in B\right\}\cap \N^p.$$
Assume that $\{\tau_1,\ldots ,\tau_t\}$ is the set of extremal rays of $\CaC_S$, and $n_i\in \tau_i$ for every $i=1, \ldots ,t$. It is well-known that $\CaC_S$ and $S$ are finitely generated if the last condition holds (see \cite[Corollary 2.10]{Brunz}). A semigroup $S$ is called simplicial if $t=p$. In general, given a cone $\CaC\subseteq \N^p$, a semigroup $S\subseteq \CaC$ is called $\CaC$-semigroup if $\CaC\setminus S$ is a finite set, with $\CaC$ equal to $\CaC_S$. Note that any numerical semigroup satisfies this property. Since the generators are coprime for numerical semigroups, we have that $\N\setminus S$ is finite. For any semigroup $S$, $\CaH(S)$ denotes the set $\CaC_S\setminus S$, whose elements are called gaps of $S$. The cardinality of its gap set is known as the genus of $S$ and is denoted by g(S). 
We call $i$-multiplicity of $S$, denoted by $\mult_i(S)$, the minimum element in $\tau_i\cap S$ for the componentwise partial order in $\N^p$. 
Other invariants require considering a monomial order on $\N^p$, a monomial order on $\N^p$ is a total order $\preceq$ on $\N^p$ satisfying compatibility with addition, and ensuring $0 \preceq c$ for any $c \in \N^p$. For example, the Frobenius element $\Fb_\preceq(S)$ of a $\CaC$-semigroup $S$ is defined as $\max_{\preceq}(\CaC \setminus S)$. To simplify the notation, fixed a monomial order $\preceq$ on $\N^p$, we use the symbol $\Fb(S)$ instead of $\Fb_\preceq(S)$. Note that $\Fb(S)$ depends on the fixed monomial order.

In this work, we explore the properties of ideals of affine semigroups; a subset $P$ of a semigroup $S$ is an ideal of $S$ if $P+S\subseteq P$, and observe that $P=S$ if and only if $0\in P$. The ideals of affine semigroups have been widely treated in the literature (to mention some of them, see \cite{Anderson1984, IdealFG, IrrIdeals}), and there exists a large list of publications devoted to the study of ideals of numerical semigroups. Some examples are \cite{Barucci, ideal R-M, CoutingMUlT2, MED_numericos} and the references therein. A recent reference is \cite{Cisto}, where the property of cofiniteness of ideals in affine semigroups is characterised. In \cite{CoutingMUlT2}, the authors introduce the concept of numerical $I(S)$-semigroup, given a numerical semigroup $S$, a numerical semigroup $P$ is called $I(S)$-semigroup if $P\setminus\{0\}$ is an ideal of $S$. Following this line of research, our work extends properties from numerical semigroups and their $I(S)$-semigroups and ideals to non-numerical affine semigroups, providing some results that are only satisfied by non-numerical affine semigroups. In particular, given $S$ an affine semigroup, we focus on characterising its affine $I(S)$-semigroups, and assuming that $S$ is a $\CaC$-semigroup, we identify those $I(S)$-semigroups that are also $\CaC$-semigroups. Moreover, we present some algorithms for computing objects related to $I(S)$-semigroups. On the one hand, for any affine semigroup $S$, we outline the description of all the $I(S)$-semigroups up to a given genus, which allows us to arrange the set of all $I(S)$-semigroups in a tree. On the other hand, for any $\CaC$-semigroup $S$, we turn out our attention on determining all the $I(S)$-semigroups with a fixed Frobenius element and a fixed set of $i$-multiplicities.

We generalise to higher dimensions the concept of numerical semigroup with maximal embedding dimension by considering ideals $M+S$ with $M=\{m_1,\ldots ,m_t\}\subset S$ such that $m_i\in \tau_i\setminus\{0\}$ for any $i\in [t]$. This new kind of affine ideals leads us to introduce affine semigroups with maximal embedding (MED-semigroup). A MED-semigroup is an affine semigroup such that all the elements in $\cap _{i=1}^t \Ap(S,n_i)\setminus\{0\}$ are minimal generators of $S$. The set $\Ap(S,m)$ denotes the Apery set of $S$ for $m\in S\setminus\{0\}$, defined as $\Ap(S,m)=\{ s\in S\mid s-m\notin S\}$. We prove that the $I(S)$-semigroup $(M+S)\cup \{0\}$ is an affine MED-semigroup. Furthermore, we characterise MED-semigroups using $I(S)$-semigroups. Our findings also provide a method for computing as many non-numerical affine MED-semigroups as desired.

Another of our objectives is to study the membership problem for an affine semigroup $S$: given an element $x$ in $\N^p$, checking whether $x$ belongs to $S$. This is an essential problem in the context of affine semigroups. Most existing methods are related to find non-negative integer solutions to some system of linear Diophantine equations. In particular, $x\in S$ means that there exist $\lambda_1,\ldots, \lambda_r\in \N$ such that $x=\sum_{i=1}^r \lambda_i n_i$. Several algorithms to find such a non-negative solution are shown in \cite{Nsemig}, and the references therein. However, the computational complexity of these methods grows with the number of variables, and the cardinality of the minimal generating set of an affine semigroup can be very large. For example, for $\CaC$-semigroups, this high cardinality can be inferred from the study made in \cite{CharactCsemig}. For a fixed numerical semigroup $S$, knowing its Apery set for one non-zero element in $S$, it is easy to solve the membership problem for $S$. Inspired by this idea, we provide an algorithm to solve this membership problem using Apery sets. Unfortunately, the Apery set of a non-numerical affine semigroup for one non-zero of its elements is not finite. However, the intersection of the Apery sets of some non-zero elements in $S$ is. We use this fact to design an algorithm for the membership problem for any non-numerical simplicial affine semigroup.

The content of this work is organised as follows: in Section \ref{IdealAffineSmgp}, we study several properties of ideals of affine semigroups and introduce the necessary background about these ideals. Sections \ref{treesection} and \ref{FrobMultSection} are devoted to improve the knowledge of $I(S)$-semigroups and to describe a tree containing all the $I(S)$-semigroups. Besides, some algorithms are given to compute the sets of all $I(S)$-semigroups up to a fixed genus, with a fixed Frobenius element, and with a set of fixed $i$-multiplicities. In Section \ref{MEDSection}, affine MED-semigroups are defined and characterised, and several families are provided. In the last section (Section \ref{AperySection}), we use some Apery sets to give an algorithm to solve the membership problem for affine semigroups. The results of this work are illustrated with several examples. To this purpose, we implemented all the algorithms shown in this work in some libraries developed by the authors in Mathematica \cite{Mathematica}.

\section{Ideals of affine semigroups}\label{IdealAffineSmgp}

Let us start by introducing some notations. Moving forward, we characterise the affine ideals and those that are also $\CaC$-semigroups. For any integers $a, b \in \N$ with $a \leq b$, we denote the integer interval $[a, b]$ as the set ${a, a+1, \ldots, b}$, and we denote the set $[n]$ as ${1, 2, \ldots, n}$.
Let $B\subset\N^p$ be a non-empty set and $x,y\in\N^p$, in this work, we consider the partial order $x\leq_B y$ if $ y-x\in B$. Given an affine semigroup $S\subset \N^p$ minimally generated by the set $\{n_1,\ldots n_t,n_{t+1},\ldots ,n_r\}$, we denote by $E$ and $A$ the sets $\{n_1,\ldots ,n_t\}$ and $\{n_{t+1},\ldots ,n_r\}$, respectively. Hereafter, we assume that $n_i=\mult_i(S)$ for any $i\in [t]$. That implies $E$ is the set of $i$-multiplicities of $S$.

Recall that a subset $P$ of an affine semigroup $S$ is an ideal of $S$ if $P+S\subseteq P$. Given a set $B$, we say that a non-empty subset $X$ of $B$ is $B$-incomparable if $x - x' \notin B$ for all  $x, x' \in X$ distinct from each other (see \cite{ideal R-M}). For instance, in the context of an affine semigroup $S$, given a non-empty subset $X$ of $S$, the set $Minimals_{\leq_S}(X)$ is $S$-incomparable.
If $X$ is a non-empty subset of an affine semigroup $S$, then $X+S$ is an ideal of $S$. We mention this case because every ideal of $S$ can be expressed in this way. This expression is not unique; that is, two different finite subsets $X_1$ and $X_2$ of $S$ could exist such that $X_1+S=X_2+S$. One such example is to consider $X_2=Minimals_{\leq_S}(X_1)$. Just as it occurs for numerical semigroups, we can achieve the desired uniqueness by imposing additional conditions. The following theorem generalises to affine semigroups Theorem 5 in \cite{ideal R-M}.

\begin{theorem}\label{thrIdealX+S}
    Let $S$ be an affine semigroup. Then,
    \[
    \{X+S\mid X \text{ is } S\text{-incomparable}\}
    \]
    is the set formed by all the ideals of $S$. Moreover, if $X_1$ and $X_2$ are different $S$-incomparable sets, then $X_1+S\ne X_2+S$.
\end{theorem}

In fact, for any ideal $P$ of an affine semigroup $S$, there exists a unique $S$-incomparable subset $X$ of $S$ such that $P=X+S$. Using the terminology given in \cite{ideal R-M}, $X$ is the ideal minimal system of generators of $P$, denoted by $imsg_S(P)$. The following lemma generalises to affine semigroups Proposition 6 in \cite{ideal R-M}.

\begin{lemma}\label{imsg}
    Let $P$ be an ideal of an affine semigroup $S$. Then, $imsg_S(P)$ is equal to $Minimals_{\leq_S}(P)$.
\end{lemma}
\begin{proof}
Let $x\in P$. Note that $x\in Minimals_{\leq_S}(P)$  if and only if there does not exist $y\in P\setminus\{x\}$ such that $x\leq_S y$, which is equivalent to $x\in imsg_S(P)$. 
\end{proof}

If $S$ is a numerical semigroup, and $P$ is an ideal of $S$, then $P\cup\{0\}$ is a numerical semigroup; the key of this result is that $\{x\in \N\mid x>\min (P)+\Fb(S)\}\subset P$. Nevertheless, the above statement extended to affine semigroups is not true.
For example, consider $S=\N^2$ and its ideal $P=\{(x,y)\in \N^2\mid x\neq 0\}$. Trivially, $P\cup\{0\}$ is not affine.  This raises the problem of determining when the ideal of a non-numerical affine semigroup provides an affine one.  The next lemma solves this question. Its proof can be obtained directly from \cite[Corollary 2.10]{Brunz}.

\begin{lemma}\label{cafin}
    Given an affine semigroup $S$, and $P$ an ideal of $S$, $P\cup\{0\}$ is an affine semigroup if and only if $\CaC_P$ is an affine cone.
\end{lemma}

From this fact, we can prove the following result.

\begin{lemma}\label{affinePcth}
    Given an affine semigroup $S$, and $P$ an ideal of $S$, then $P\cup\{0\}$ is an affine semigroup and $\CaC_P=\CaC_S$ if and only if there exists $Y\subset P\setminus\{0\}$ such that $Y\cap \tau_i\ne \emptyset$ for all $i\in[t]$. 
\end{lemma}
\begin{proof}
    If $P\cup\{0\}$ is an affine semigroup and $\CaC_P=\CaC_S$, then $P\cap \tau_i\ne \emptyset$ for all $i\in[t]$. It is enough to take $Y=\{y_1,\ldots, y_t\}$, with $y_i\notin P\cap \tau_i$. Conversely, if there exists $Y\subset P\setminus\{0\}$ such that, $Y\cap \tau_i\ne \emptyset$ for all $i\in[t]$, then $\CaC_P=\CaC_S$. By applying Lemma \ref{cafin}, $P\cup\{0\}$ is an affine semigroup.
\end{proof}

Recall that an affine semigroup $S$ is a $\CaC$-semigroup if the complement of $S$ in $\CaC_S$ is finite. Lemmas \ref{cafin} and \ref{affinePcth} can be refined to $\CaC$-semigroups. Given any subset $B\subset \N^p$, $\langle B\rangle$ denotes the set $\left\{ \sum _{i=1}^k \lambda_i b_i\mid k, \lambda_1,\ldots ,\lambda_k\in \N,\mbox{ and } b_1,\ldots ,b_k\in B\right\}$.

\begin{lemma}\label{lema_tecnico}
    Let $S$ be a $\CaC$-semigroup, and let $P$ be an ideal of $S$. If $P\cup\{0\}$ is a $\CaC$-semigroup, then $\CaC_P=\CaC_S$.
\end{lemma}

\begin{proof}
    Let $P$ be an ideal of $S$, then $\CaC_P\subseteq\CaC_S$. If $\CaC_P\neq\CaC_S$, then there exists $\tau\in\CaC_S$ such that $\tau\notin\CaC_P=\langle \alpha_1,\ldots,\alpha_r\rangle$. Consider $x\in P$, since $P$ is an ideal of $S$, we obtain that for every $k\in \N$, $x+k\tau=\sum_{i=1}^r\lambda_i\alpha_i$, for some $\lambda_1,\ldots,\lambda_r\in \N$.
    Taking into account $x= \sum_{i=1}^r\gamma_i\alpha_i$ for some $\gamma_1,\ldots,\gamma_r\in \N$, and by choosing $k$ large enough such that $\lambda_i>\gamma_i$ for all $i\in [r]$, it follows that $k\tau=\sum_{i=1}^r(\lambda_i-\gamma_i)\alpha_i$, contradicting that $\tau\notin\CaC_P$.
\end{proof}

\begin{proposition}\label{Pideal<->PC-smgp}
    Let $S$ be a $\CaC$-semigroup, and let $P$ be an ideal of $S$. Then, $P\cup \{0\}$ is a $\CaC$-semigroup if and only if there exists a finite set $Y\subset P\setminus\{0\}$ such that $Y\cap \tau_i\ne \emptyset$, for all $i\in[t]$.
\end{proposition}

\begin{proof}
In view of Lemma \ref{lema_tecnico}, if $P\cup \{0\}$ is a $\CaC$-semigroup, then there exists $Y\subset P\setminus\{0\}$ such that $Y\cap \tau_i\ne \emptyset$, for any $i\in[t]$.
Conversely, assume that there exists a finite subset $Y$ of $P\setminus\{0\}$ such that, $Y\cap \tau_i\ne \emptyset$ for all $i\in[t]$. By applying Lemma \ref{affinePcth}, $\CaC_P=\CaC_S$, and $P\cup \{0\}$ is an affine semigroup. We point out that $Y+S\subseteq P$. To prove that $\CaH(P)$ is finite, and taking into account that $\CaC_S=S\sqcup\CaH(S)$ where $\CaH(S)$ is a finite subset, it suffices to show that the cardinality of $\CaC_S\setminus\left(Y+\CaC_S\right)$ is finite ($\sharp(B)$ denotes the cardinality of any set $B$). Let $\CaC_S=\langle b_1,\ldots ,b_m\rangle$, then, for each $i\in[m]$, there exists a positive integer $k_i$ such that $k_ib_i\in \langle Y\rangle$. Since $x\in \CaC_S$ if and only if $x=\sum_{i=1}^{m}\lambda_ib_i$, thus, $\sharp \left(\CaC_S\setminus\big(Y+\CaC_S\big)\right)=\sharp \{\sum_{i=1}^{m}\lambda_ib_i\mid \lambda_i\leq k_i \text{ for all } i\in[m]\}< \infty$. Hence, $P\cup \{0\}$ is a $\CaC$-semigroup.
\end{proof}

\section{$I(S)$-semigroups and its associated tree}\label{treesection}

Let $S$ be an affine semigroup. Recall that an affine semigroup $T$ is an $I(S)$-semigroup of $S$ if $T\setminus\{0\}$ is an ideal of $S$. We can easily rewrite Proposition \ref{Pideal<->PC-smgp}  from this definition. 

\begin{corollary}\label{I(S)C-smp}
   Let $S$ be a $\CaC$-semigroup, and let $T$ be an $I(S)$-semigroup. Then, $T$ is a $\CaC$-semigroup if and only if there exists $Y\subset T\setminus\{0\}$ such that $Y\cap \tau_i\ne \emptyset$, for all $i\in[t]$.
\end{corollary}

As an immediate consequence of Corollary \ref{I(S)C-smp}, the analysis of $I(S)$-semigroups is equivalent to the study of ideals of $S$ meeting the conditions outlined in Proposition \ref{Pideal<->PC-smgp}. Furthermore, as observed in numerical semigroups (see \cite{ideal R-M}), given a $\CaC$-semigroup $S$, $T$ is an $I(S)$-semigroup if and only if $T\subseteq S \subseteq T\cup PF(T)$, where $PF(S)$ is the set of pseudo-Frobenius elements defined as $\{x\in \CaH(S)\mid x+(S\setminus\{0\})\subset S\}$. In the context of $I(S)$-semigroups, we interpret Lemma \ref{imsg} as follows.

\begin{lemma}
    Let $S$ be an affine semigroup and let $T$ be an $I(S)$-semigroup. Then, $imsg_S(T\setminus\{0\})$ is finite. Moreover, $imsg_S(T\setminus\{0\})=Minimals_{\leq_S}(msg(T))$.
\end{lemma}

\begin{proof}
    By applying Lemma \ref{imsg} it is suffices to prove $ Minimals_{\leq_S}(msg(T))\subseteq Minimals_{\leq_S}(T\setminus\{0\})$. Let $x\in T\setminus\{0\}$ such that $x=y+s$, for some $s\in S$ and $y\in T\setminus\{0\}$. If $x\notin msg(T)$, then clearly $x\notin Minimals_{\leq_S}(msg(T))$. If $x\in msg(T)$, then $s\in S\setminus T$ and necessarily $y\in msg(T)$ and thus $x\notin Minimals_{\leq_S}(msg(T))$.
\end{proof}

This section shows how the set of all $I(S)$-semigroups can be arranged in a tree, drawing inspiration from \cite{ideal R-M}. From now on, we fixed a monomial order $\preceq$ on $\N ^p$, and let $\mathcal{J(S)}=\{T\mid T \text{ is an $I(S)$-semigroup}\}$.

The following results are essential to obtain the announcement tree, and it has a straightforward proof (see \cite[Lemma 27]{ideal R-M}). Given $A$ and $B$ two subsets of $\N^p$ such that $A \setminus B$ is finite, $\mathcal{O}_A(B)$ is defined as $\max_\preceq(A\setminus B)$.

\begin{lemma}\label{uptree}
    Let $S$ be a $\CaC$-semigroup and $T$ be a non-proper $I(S)$-semigroup. Then, $T \cup \{\mathcal{O}_S(T)\}\in \mathcal{J(S)}$. 
\end{lemma}

We define $G(\mathcal{J(S)})$, the associated graph to $\mathcal{J(S)}$, in the
following way: the set of vertices of $G(\mathcal{J(S)})$ is $\mathcal{J(S)}$ and $(T_1,T_2) \in \mathcal{J(S)} \times \mathcal{J(S)}$ is an edge if $T_2 = T_1 \cup \{\mathcal{O}_S(T_1)\}$. When $(T_1, T_2)$ is an edge, we say that $T_1$ is a child of $T_2$. 

From Lemma 8 in \cite{AffineConBody}, we deduce that given an affine semigroup $S$ and an element $x$ of $S$, then $S\setminus\{x\}$ is an affine semigroup if and only if $x\in msg(S)$. This characterisation can be translated to $I(S)$-semigroup as follows.

\begin{lemma}\label{downtree}
    Let $S$ be an affine semigroup, $T$ be an $I(S)$-semigroup, and $x\in msg(T)$. Then $T\setminus\{x\}$ is an $I(S)$-semigroup if and only if $x\in imsg_S(T\setminus\{0\})$.
\end{lemma}

\begin{proof}
This is followed by arguing as in \cite[Lemma 33]{ideal R-M}.    
\end{proof}

\begin{theorem}\label{thrmTree}
    For any  $\CaC$-semigroup $S$, $G(\mathcal{J(S)})$ is a tree with root $S$. Furthermore, the set of children of any $T\in \mathcal{J(S)}$ is the set 
    \[\{T\setminus\{x\}\mid x\in imsg_S(T\setminus\{0\}) \text{ and } x\succ \mathcal{O}_S(T)\}.\]
\end{theorem}
\begin{proof}
    Let $T\in \mathcal{J(S)}$. We 
    consider the sequence of $I(S)$-semigroups $\{T_i\}_{i\in \N}$ defined by $T_{i+1}=T_i\cup\{\mathcal{O}_S(T_i)\}$. Considering that $S\setminus T$ is finite, 
    a unique path exists connecting $T$ with $S$, defined using the above sequence. Regarding the second assertion. If $B$ a child of $T$ then, $T=B\cup\{\mathcal{O}_S(B)\}$, thus by Lemma \ref{downtree}, $B=T\setminus \{\mathcal{O}_S(B)\}$ is an $I(S)$-semigroup and, $\mathcal{O}_S(B)\in imsg_S(T\setminus\{0\}) $. By definition, $\mathcal{O}_S(B)\ne\mathcal{O}_S(T)$. If $\mathcal{O}_S(B)\prec\mathcal{O}_S(T)$, then by the maximality of $B$, $\mathcal{O}_S(T)\in B$, which it is not possible. Conversely, suppose that $B=T\setminus\{x\}$ is an $I(S)$-semigroup with $x\succ \mathcal{O}_S(T)$, whence $\mathcal{O}_S(B)=\max_\preceq(S\setminus B)=\max_\preceq\{\max_\preceq(S\setminus T),x\}=x.$
\end{proof}

The above result can be used to recurrently build $G(\mathcal{J(S)})$. Let us prove that it is indeed infinite. 

\begin{proposition}
    Let $S$ be a $\CaC$-semigroup, and $g$ be an integer greater than or equal to $g(S)$. Then, at least one $I(S)$-semigroup $T$ with genus $g$ exists.
\end{proposition}
\begin{proof}
    Let $E_0=\cup_{i\in [t]} \mult_i(S)=\{n_1,\ldots, n_t\}$, and $T_0=E_0+S$ be an $I(S)$-semigroup. If $g(T_0)= g$ we have already finished. Otherwise $g(T_0)< g$, we consider the sequence of $I(S)$-semigroups $\{T_j\}_{j\in \N}$ defined by $T_j=E_j+S$, where
    $$E_j=\Big\{ n_1,\ldots, n_{k-1}, \min_\preceq\{(S\cap \tau_k)\setminus E_{j-1}\} , n_{k+1},\ldots ,n_t\Big\},$$
    for a fixed $k\in [t]$. Hence, a positive integer $j_0$ exists such that $g(T_{j_0})= g$. Applying Lemma \ref{uptree}  $g-g(T_{j_0})$ times, we obtain an $I(S)$-semigroup $T$ with genus $g$.
\end{proof}

Let $S$ be a $\CaC$-semigroup, and let $g$ be an integer such that $g\geq g(S)$. Consider the set $\mathcal{J}(S)_g=\{T \mid T\text{ is an $I(S)$-semigroup with } g(T)\leq g\}$. As a consequence of Theorem \ref{thrmTree}, Algorithm \ref{alg} computes $G(\mathcal{J}(S)_g)$.

\begin{algorithm}
\caption{Computing all $I(S)$-semigroups of genus up to $g\geq g(S)$.}\label{alg}
\KwIn{Let $S$ be a $\CaC$-semigroup and an integer $g$ such that $g\geq g(S)$.}
\KwOut{The set $\{T\mid T\text{ is an $I(S)$-semigroup with } g(T)\leq g\}$.}
\If {$g(S)=g$}
    {\Return{$S$}}
$I \leftarrow \{S\}$\;
$X \leftarrow \emptyset$\;
\For{$i\in [g(S),g]$}{
    $Y \leftarrow \emptyset$\;
    \While {$I\ne\emptyset$}{
        $T \leftarrow \text{First}(I)$\;
        $B_T \leftarrow \{x\in imsg_S(T\setminus\{0\})\mid x\succ \mathcal{O}_S(T)\}$\;
        $Y \leftarrow Y\cup\{T\setminus\{x\}\mid x\in B_T\}$\;
        $I \leftarrow I\setminus \{T\}$\;
    }
    $X \leftarrow  X\cup \{Y\}$\;
    $I \leftarrow Y$\;
    }
    \Return{$X$}
\end{algorithm}

The next example illustrates Algorithm \ref{alg}.

\begin{example}\label{example_tree}
Let $S$ be the $\CaC$-semigroup with genus 4, and minimally generated by the set $\{(5, 1), (6, 2), (9, 2), (9, 3), (10, 3), (12, 3), (13, 4), (13, 3)\}$. Applying Algorithm \ref{alg} to $S$, we obtain that the amount of $I(S)$-semigroups up to genus 6 is 51. The tree $G(\mathcal{J}(S)_6)$ is shown in Figure \ref{fig:tree}. To ensure more clarity in the figure, each vertex of the tree is labelled with the element removed to reach its parent node.
\end{example}
\begin{figure}[h]
    \includegraphics[scale=0.49,angle=90]{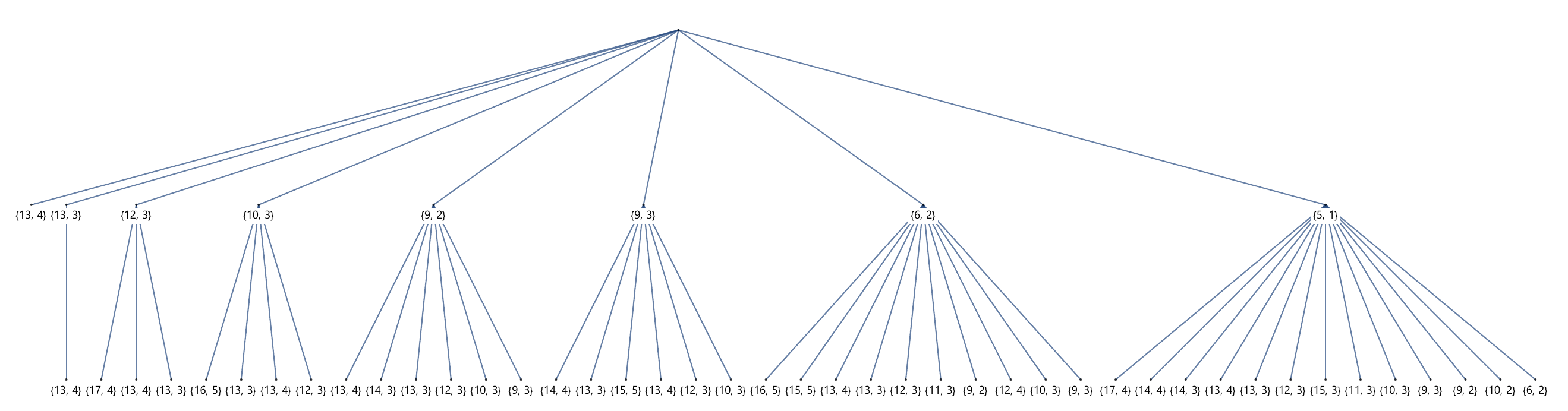}
    \caption{The tree $G(\mathcal{J}(S)_6)$.}
    \label{fig:tree}
\end{figure}

\section{Computing $I(S)$-semigroups with a fixed Frobenius element and i-multiplicities}\label{FrobMultSection}

This section assumes that any $I(S)$-semigroup is a $\CaC$-semigroup. One objective of this section is to explicitly describe all the $I(S)$-semigroups with a fixed Frobenius element. Given $f\in \CaC$ and a monomial order $\preceq$, consider $A_S(f)=\{x\in S\mid x\prec f \text{ and } f-x\notin S\}$. 

\begin{proposition}
    Let $S$ be a $\CaC$-semigroup, $f\in \CaC$ greater than or equal to $\Fb(S)$, and $\preceq$ a monomial order. The following conditions are equivalent:
    \begin{enumerate}
        \item $T$ is an $I(S)$-semigroup with Frobenius element $\Fb(T)=f$.
        \item $T=X\sqcup\{x\in \CaC\mid x\succ f\}\cup\{0\}$, where $X$ is a subset of $A_S(f)$ such that if there exists $x\in X$ with $x+s\prec f$ for some $s\in S$, then $x+s\in X$.
    \end{enumerate}
\end{proposition}

\begin{proof}
    If $T$ is an $I(S)$-semigroup with Frobenius element $\Fb(T)=f\succeq \Fb(S)$, then $\{x\in \CaC\mid x\succ f\}\cup\{0\}\subseteq T$. Let $x\in T$ such that $x\prec f$, it follows that $x\in A_S(f)$, otherwise $f-x\in S$, and since $T\setminus\{0\}$ is an ideal of $S$, we have that $x+(f-x)=f\in T\setminus\{0\}$, contradicting that $\Fb(T)=f$. So, $T=X\sqcup\{x\in \CaC\mid x\succ f\}\cup\{0\}$ where $X= T\cap A_S(f)$.
    
    Conversely, let us show that $T\setminus\{0\}=X\sqcup\{x\in \CaC\mid x\succ f\}$ is an ideal of $S$. Trivially, $T\setminus\{0\}\subseteq S$. Let $x\in T\setminus\{0\}$ and $s\in S$, we assume that $x\prec f$, otherwise the result is clear. Thus, $x\in X$. We can differentiate cases based on $x+s$. Notably, $x+s\neq f$ since $x\in X$. If $x+s\succ f$, then $x+s\in T\setminus\{0\}$. If $x+s\prec f$, then, by hypothesis, $x+s\in X$. Thus, $T+S\subseteq T$.
\end{proof}

The above result provides an algorithm for computing all the $I(S)$-semigroups whose Frobenius element equals $f$. Given any subset $A$, we denote by $\mathcal{P}(A)$ the set of all possible subsets of $A$.

\begin{algorithm}[H]
\caption{Computing all $I(S)$-semigroups with a fixed Frobenius element.}\label{computeI(S)fixedFb}
\KwIn{A $\CaC$-semigroup $S$, and $f\in \CaC\setminus\{0\}$.}
\KwOut{$\{T\mid T \text{ is an }I(S) \text{-semigroup with } \Fb(T)=f\}$.}
\If{$f\prec \Fb(S)$}
    {\Return $\emptyset$}
$L \leftarrow \{s\in S\mid s\prec f\}$\\
$B \leftarrow \{x\in L\mid f-x\notin S\} $\\
$P \leftarrow \mathcal{P}(B)$\\
$\mathcal R\leftarrow \emptyset$\\

\While{$P\neq \emptyset$}
    {$X \leftarrow First(P)$\\
     \If{$x+s\in X$, for all $x\in X$ and $s\in L$ such that $x+s\prec f$}
        {$\mathcal R= \mathcal  R \cup \big(X \sqcup \{x\in \CaC\mid x\succ f\}\cup\{0\}\big)$}
    $P \leftarrow P\setminus \{X\}$\\
    }

\Return{$\mathcal R$}
\end{algorithm}

\begin{example}
    Let $S$ be the $\CaC$-semigroup appearing in Example \ref{example_tree}, and consider $f=(11,3)$. Applying Algorithm \ref{computeI(S)fixedFb} with the degree lexicographic order fixed, we obtain that all $I(S)$-semigroups with Frobenius element $f=(11,3)$ are determined by $X \sqcup \{x\in \CaC\mid x\succ (11,3)\}\cup\{0\}$ for each $X$ in the set
    \begin{multline*}
        \big\{\{\emptyset\}, \{(9, 2)\}, \{(9, 3)\}, \{(10, 2)\}, \{(10, 3)\}, \{(9, 2), (9, 3)\}, \{(9, 2), (10, 2)\}, \\
        \{(9, 2), (10, 3)\}, \{(9, 3), (10, 2)\}, \{(9, 3), (10, 
   3)\}, \{(10, 2), (10, 3)\}, \{(9, 2), (9, 3), (10, 2)\},\\
   \{(9, 2), (9, 3), (10, 3)\}, \{(9, 2), (10, 2), (10, 3)\},
   \{(9, 3), (10, 2), (10, 3)\},\\ \{(9, 2), (9, 3), (10, 2), (10, 3)\}\big\}.
    \end{multline*}
\end{example}

The other aim of this section is to propose an algorithm to compute all the $I(S)$-semigroups fixed the multiplicity of each external ray. To show it, we need the following technical lemma.
\begin{lemma}
    Let $S$ be a $\CaC$-semigroup, and $M=\{m_1,\ldots, m_t\}\subset S\setminus\{0\}$ such that  $m_i\in\tau_i\setminus\{0\}$ for all $i\in[t]$. Then, the set 
    $$\mathcal T=\{T \text{ is an } I(S)\text{-semigroup }\mid \cup_{i\in[t]} \mult_i(T)=M\}$$
    is a non-empty finite set.
\end{lemma}
\begin{proof}
    Let $T$ be an $I(S)$-semigroup with $\cup_{i\in[t]} \mult_i(T)=M$. By applying Theorem \ref{thrIdealX+S}, $T\setminus\{0\}=imsg_S(T\setminus\{0\})+S$, and thus $imsg_S(T\setminus\{0\})=M\sqcup X$, where $X$ is a subset of $S\setminus((M+S)\cup \{0\})$. As $S\setminus(M+S)$ is finite, we deduce that there exists a finite amount of $I(S)$-semigroups $T$ with $\cup_{i\in[t]}\mult_i(T)=M$.

    Note that $\mathcal T$ is not empty since $(M+S)\cup \{0\}$ is an $I(S)$-semigroup satisfying that $\cup_{i\in[t]}\mult_i((M+S)\cup \{0\})=M$.
\end{proof}

We are now in a position to describe the proposed algorithm. This algorithm is directly deduced from the proof of the above lemma.

\begin{algorithm}[H]
\caption{Computing all $I(S)$-semigroups with a fixed multiplicity of each external ray.}\label{computeI(S)fixedM}
\KwIn{A $\CaC$-semigroup $S$ and $M=\{m_1,\ldots, m_t\}\subset S\setminus\{0\}$ such that  $m_i\in\tau_i\setminus\{0\}$ for all $i\in[t]$.}
\KwOut{$\{T\mid T \text{ is a }\CaC \text{-semigroup with } \cup_{i\in[t]}\mult_i(T)=\{m_1,\ldots,m_t\}\}$.}

$H \leftarrow \CaH\big( (M+S)\cup \{0\}\big)$\\

$B \leftarrow H\cap S$\\
\Return{$\left\{\big((M\sqcup X)+S\big)\cup \{0\}\mid X\in \mathcal P (B) \right\}$}
\end{algorithm}

\begin{example}
Again, let $S$ be the $\CaC$-semigroup introduced in Example \ref{example_tree}, and consider $M=\{(10,2),(6,2)\}$. Algorithm \ref{computeI(S)fixedM} computes the $2047$ sets $X$ determining all $I(S)$-semigroups with the $i$-multiplicities in $M$, but not all of these $I(S)$-semigroups are different. For example, for the set
$$B=\{(5,1),(9,2),(9,3),(10,3),(12,3),(13,3),(13,4),(14,3),(14,4),(17,4),(18,4)\}$$
obtained, the sets
$$X_1=\{(5, 1), (9, 2), (9, 3), (10, 3), (12, 3), (13, 3), (13, 4), (14,  3), (14, 4), (17, 4)\},$$
and
$$X_2=\{(5, 1), (9, 2), (9, 3), (10, 3), (12,  3), (13, 3), (13, 4), (14, 3), (14, 4), (18, 4)\}$$
belong to $\mathcal P (B)$. As the $S$-incomparable sets of $M\sqcup X_1$ and $M\sqcup X_2$ are the same set $$\{(5, 1), (6, 2), (9, 2), (9, 3), (10, 3), (12, 3), (13, 3), (13, 4)\},$$ 
according to Theorem \ref{thrIdealX+S}, we know that both sets yield the same $I(S)$-semigroup. Algorithm \ref{computeI(S)fixedM} computes the $351$ $I(S)$-semigroups such that $M$ is the union of the $i$-multiplicities of any of them.
\end{example}

\section{Ideals of semigroups and affine MED-semigroups}\label{MEDSection}

From now on, let $S\subset \N^p$ be an affine semigroup with $t$ extremal rays, and minimally generated by $E\sqcup A$ with $E=\cup_{i\in [t]} \mult_i (S)=\{n_1,\ldots ,n_t\}$, and $A=\{n_{t+1},\ldots ,n_r\}$. In this section, we consider the ideals of affine semigroups $M+S$ with $M=\{m_1,\ldots ,m_t\}$ a finite subset of $S$ such that $m_i\in \tau_i\setminus\{0\}$ for any $i\in [t]$. This kind of ideals allows us to generalise the concept of maximal embedding dimension from numerical semigroup to affine semigroups. Let us start with some necessary definitions and results.

Recall that the Apery set of $S$ respect to $m\in S$ is the set $\Ap(S,m)=\{ s\in S\mid s-m\notin S\}$. For any non-numerical affine semigroup, this set is not finite, but the intersection $\cap _{i\in [t]} \Ap(S,p_i)$ is finite for any fixed elements $p_i\in \tau_i\cap S$. Note that $A\subset \cap _{i\in [t]} \Ap(S,n_i)$, that is, $E\sqcup \big(\cap _{i\in [t]} \Ap(S,n_i)\setminus\{0\}\big)$ is a generating set of $S$.

\begin{definition}
    Given $S\subset \N^p$ an affine semigroup minimally generated by $E\sqcup A$, $S$ is a maximal embedding dimension affine semigroup (MED-semigroup) if $\cap _{i\in [t]} \Ap(S,n_i)= A\sqcup \{0\}$.
\end{definition}

These semigroups can be characterised by their minimal generating sets.

\begin{proposition}\label{caracterizacion_MED}
The affine semigroup $S$ is a MED-semigroup if and only if for any $i,j\in [t+1,r]$, there exists $k\in [t]$ such that $n_i+n_j-n_k\in S$.
\end{proposition}

\begin{proof}
    Assume that $S$ is a MED-semigroup. Thus, any non-zero $m\in S$ satisfies that $m-n_k\notin S$ for every $k\in [t]$ if and only if $m\in A$. Since the elements in $A$ are all minimal generators, we have that $n_i+n_j\notin \cap _{k\in [t]} \Ap(S,n_k)$ for any $i,j\in [t+1,r]$.

    Conversely, let $m$ be an element belonging to $\cap _{i\in [t]} \Ap(S,n_i)$, and consider that for all $i,j\in [t+1,r]$ there exists at least an integer $k\in [t]$ such that $n_i+n_j-n_k\in S$. Hence, $m=\sum_{q=t+1}^r\lambda _q n_q$ for some $\lambda _{t+1},\ldots ,\lambda _{r}\in \N$. Our hypothesis means that $\sum_{q=t+1}^r\lambda _q=1$, and so Proposition holds.
\end{proof}

The following result determines the relationship between affine MED-semigroups and the affine ideals. 
In addition, this lemma provides a method to construct an arbitrary number of affine MED-semigroups.

\begin{lemma}\label{MED_E_plus_S}
Let $S$ be an affine semigroup, $M=\{m_1,\ldots ,m_t\}\subset S$ such that $M\cap (\tau_i \setminus \{0\})\neq \emptyset$ for every $i\in [t]$, and the ideal of $S$ defined by $M+S$. Then, the $I(S)$-semigroup $T=(M+S)\cup \{0\}$ is a MED-semigroup. Moreover, $$\CaH(T)=\CaH(S)\sqcup \left(\cap _{i\in [t]} \Ap(S,m_i)\setminus\{0\}\right).$$
\end{lemma}

\begin{proof}
    By construction, $M$ is a subset of the minimal generating set of $T$, and $M\cap \big(\cap _{i\in [t]} \Ap(T,m_i)\big)$ is the empty set. Let $x\neq 0$ be an element of $\cap _{i\in [t]} \Ap(T,m_i)$. Suppose $x$ is not a minimal generator of $T$, that is, $x=y+z$ for some non-null $y,z\in T$. Therefore, $x=m_j+s+m_k+s'$ for some  $j,k\in [t]$, and $s,s'\in S$. Thus, $x-m_j\in M+S\subset T$, which is not possible since $x\in \cap _{i\in [t]} \Ap(T,m_i)$. Hence, $M\sqcup \left(\cap _{i\in [t]} \Ap(T,m_i)\setminus\{0\}\right)$ is the minimal generating set of $T$. We conclude that $T$ is a MED-semigroup.

    Since $T\subset S$, $\CaH(S)\subset \CaH(T)$. Consider $x\in \cap _{i\in [t]} \Ap(S,m_i)\setminus\{0\}$. Hence, $x-m_i\notin S$ for any $i\in [t]$. This means that $x\in \CaH(T)$, and $\CaH(S)\sqcup \left(\cap _{i\in [t]} \Ap(S,m_i)\setminus\{0\}\right)$ is a subset of $\CaH(T)$. Let $x$ be an element in $\CaH(T)$. If $x\notin S$, then $x\in \CaH(S)$. In the other case, if $x-m_i\in S$ for some $i\in [t]$, then $x\in T$, which is impossible. So, $x\in \cap _{i\in [t]} \Ap(S,m_i)\setminus\{0\}$.
\end{proof}

The above lemma can be used to compute some steps in Algorithm \ref{computeI(S)fixedM}. Besides, some useful results are obtained.

\begin{proposition}\label{proposition_E_plus_S}
It holds that:
    \begin{itemize}
        \item $S$ is a $\CaC$-semigroup if and only if $T$ is a $\CaC$-semigroup.
        \item $T=\left(S\setminus  \cap _{i\in [t]} \Ap(S,m_i) \right)\cup \{0\}$.
    \end{itemize}
\end{proposition}

To use the $I(S)$-semigroup $T=(M+S)\cup \{0\}$ in a computational way, it is necessary to know a generating set of $T$. One such generating set is determined from a finite set denoted by $\Gamma$. To do that, we take into account that, for any $n_j\in E\sqcup A$, there exists a minimum non-zero integer $q_j$ such that $q_jn_j$ is equal to $\sum _{i=1}^t \mu_i m_i$ for some integers $\mu_i\in \N$. So, we consider the set 
\begin{equation}\label{gamma}
    \Gamma=\left\{\sum_{j=1}^r \lambda_j n_j\mid \lambda_j\in [0,q_j-1]\right\}.
\end{equation}

\begin{lemma}\label{generadores_E_plus_S}
    Let $S$ be an affine semigroup and $M=\{m_1,\ldots ,m_t\}\subset S$ such that $m_i\in \tau_i\setminus\{0\}$ for any $i\in [t]$. The set $\bigcup _{i\in [t]}\{m_i+\gamma \mid \gamma \in \Gamma \}$ is a system of generators of the $I(S)$-semigroup $T=(M+S)\cup \{0\}$.
\end{lemma}

\begin{proof}
    By construction, the affine semigroup generated by $\bigcup _{i\in [t]}\{m_i+\gamma \mid \gamma \in \Gamma \}$ is a subset of $T$. Consider $x\in T$, hence there are some $\lambda_1,\ldots ,\lambda_r\in \N$, and $i\in [t]$ such that $x= m_i+\sum_{j=1}^r \lambda _jn_j$. If for some $k\in [r]$, $\lambda_k\ge q_k$, where $q_k$ is determined from the definition of $\Gamma$, then $\lambda_kn_k = \sum_{h=1}^t \mu_hm_h$ with $\mu_h\in \N$ for every $h\in [t]$. Therefore, $x=\sum_{i=1}^t \nu_i m_i+ \sum_{i=1}^r \xi_i n_i$ for some $\nu_1,\ldots ,\nu_t\in \N$, and $\xi_1,\ldots ,\xi_r\in [0,q_1-1]\times \cdots \times [0,q_r-1]$. Thus, the lemma holds.
\end{proof}

From the previous results, we show an example of a MED-semigroup constructed using Lemma \ref{MED_E_plus_S}.

\begin{example}\label{MED_example2}
Let $S\subset \N^2$ be the affine semigroup minimally generated by
$$\{ (5, 1), (6, 2),  (8, 2), (9, 2), (12, 3)\},$$
and $M=\{ (5, 1), (6, 2)\}$. We obtain that the set $\Gamma$ is equal to 
\begin{multline}\label{gamma_example}
    \{(0,0),(8,2),(9,2),(12,3),(17,4),(18,4),(20,5),(21,5),(24,6),(26,6),\\(27,6), (29,7), (30,7),(32,8),(33,8),(35,8),(36,9),(38,9),(39,9),\\(41,10), (42,10),(44,11),(45,11),(47,11),(50,12),(51,12),(53,13),\\(54,13), (59,14),(62,15), (63,15),(71,17)\}.
\end{multline}
We consider $T=(M+S)\cup \{0\}$ and compute its minimal generating set from the set $\bigcup _{i\in M}\{m_i+\gamma \mid \gamma \in \Gamma \}$ determined in Lemma \ref{generadores_E_plus_S}. Hence, the minimal generating set of $T$ is 
$$\{(5, 1), (6, 2), (13, 3), (14, 3), (14, 4), (15, 4), (17, 4), (18, 5)\}.$$ 
Figure \ref{MED_Eplus_S} gives us a graphical representation of $T$. The empty circles are the gaps of $T$, and the full red circles are elements of $T$.
\begin{figure}[H]
    \centering
    \includegraphics[scale=.46]{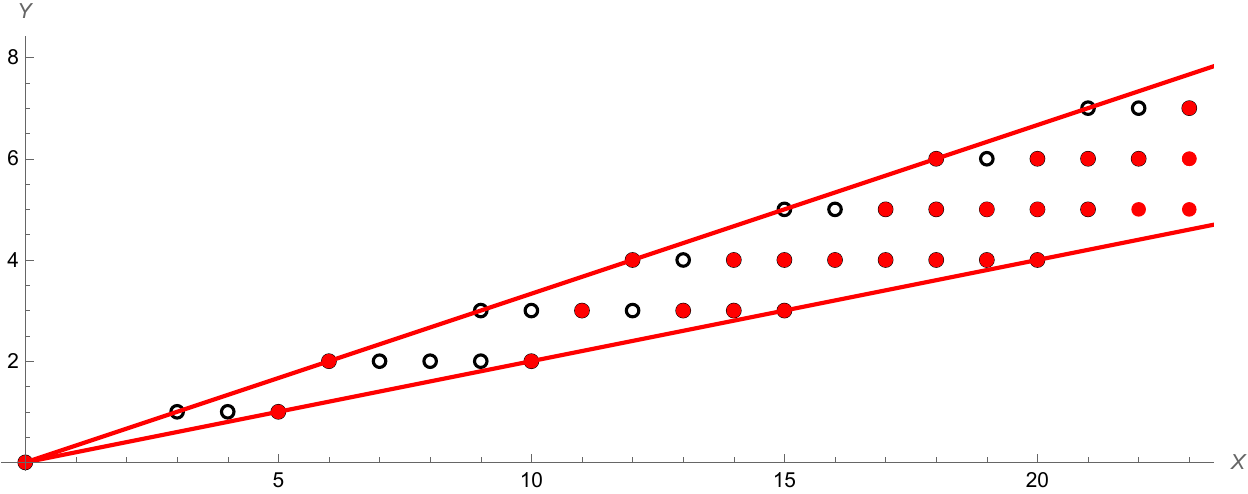}
    \caption{MED-semigroup $T=(E+S)\cup \{0\}$.}
    \label{MED_Eplus_S}
\end{figure}
\end{example}

Numerical MED-semigroups have a concrete structure. In \cite{MED_numericos}, it is proved that a numerical semigroup $T$ is a MED-semigroup if and only if there exist a numerical semigroup $S$ and $m \in S \setminus\{0\}$ such that $T = (m +S)\cup \{0\}$. From the above statement, a natural question is born: is Lemma \ref{MED_E_plus_S} the natural generalisation of the numerical case? The following example answers that it is not true.

\begin{example}
Let $S$ be again the MED-semigroup considered in Example \ref{MED_example2}, which is shown in Figure \ref{MED_decomposition}. As mentioned earlier, the empty circles are the gaps of $T$, and the filled red circles are elements of $T$.
\begin{figure}[H]
    \centering
    \includegraphics[scale=.48]{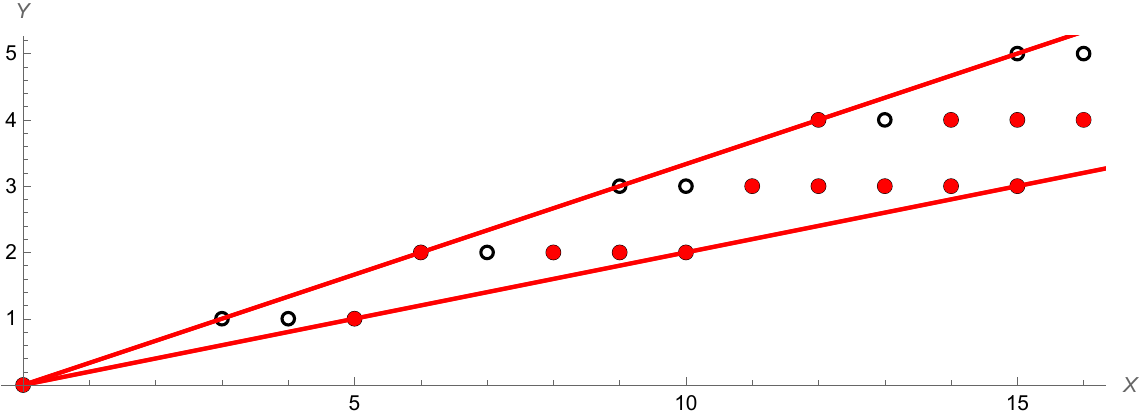}
    \caption{MED-semigroup generated by $\{ (5, 1), (6, 2),  (8, 2), (9, 2), (12, 3)\}$.}
    \label{MED_decomposition}
\end{figure}
Its minimal generating set is $E\sqcup A$, with $E=\{(5, 1), (6, 2)\}$, and $A=\{(8, 2), (9, 2), (12, 3)\}$. 
Let us prove that $S$ is not equal to $M+H$ with $H$ an affine semigroup and $M=\{m_1,m_2\}\subset S\setminus\{0\}$ satisfying $m_i\in \tau_i$ for $i=1,2$. Suppose there exists an affine semigroup $H$ and a set $M$ satisfying the specified conditions. So, $m_1$ has to be equal to $a(5, 1)$, and $m_2=b(6, 2)$ for some non-zero $a,b\in \N$. Thus, we deduce that either $(8,2)-a(5, 1)\in H$, or $(8,2)-b(6, 2)\in H$. Since $(8,2)-b(6, 2)\notin \CaC_S$, $(8,2)-b(6, 2)\notin H$. Analogously, for any $a\geq 2$, $(8,2)-a(5, 1)\notin \CaC_S$. If $a=1$ and assume $(8,2)-1(5, 1)=(3,1)\in H$, then $(6,2)+(3,1)\in S$, which it is not true. Hence, the MED-semigroup $S$ cannot be obtained from Lemma \ref{MED_E_plus_S}.
\end{example}

In spite of Lemma \ref{MED_E_plus_S} does not determine all non-numerical affine MED-semigroups, it induces, together with Proposition \ref{proposition_E_plus_S}, a characterisation of them.

\begin{theorem}
   Let $S$ be a semigroup minimally generated by $E\sqcup A$. Then, $S$ is MED-semigroup if and only if the $I(S)$-semigroup $(E+S)\cup \{0\}$ is equal to $S\setminus A$.
\end{theorem}

As we have seen, the structure of general affine MED-semigroups is more complex than numerical semigroups. Now, we present a decomposition using some ideals of an affine semigroup, which can be applied to any affine semigroup, to introduce a different family of affine MED-semigroups. Hereafter, consider the finite set $S_0=\cap_{i\in [t]} \{s\in S\mid s-n_i\notin \CaC_S\}$, and $S_i=  \{x\in \CaC_S\mid x+n_i\in S\}$ for all $i\in [t]$. Note that each $n_i+S_i$ is an ideal of $S$, and thus $\bigcup_{i\in [t]} (n_i+S_i)$ is also an ideal of $S$. 

\begin{lemma}\label{lemma_decomposition}
    Let $S$ be a semigroup minimally generated by $E\sqcup A$ with $E=\{n_1,\ldots, n_t\}$. Then, $S=S_0\cup \bigcup_{i\in [t]} (n_i+S_i)$.
\end{lemma}

\begin{proof}
By definitions, $S_0\cup \bigcup_{i\in [t]} (n_i+S_i)\subseteq S$. Thus, it is enough to prove the opposite inclusion. Given $x\in S$, there are two possibilities, either $x-n_i\notin \CaC_S$ for all $i\in [t]$, or there exists $k\in [t]$ such that $x-n_k\in \CaC_S$. If the first case holds, then $x\in S_0$. In the other one, $x\in n_k+S_k$. So, $x\in S_0\cup \bigcup_{i\in [t]} (n_i+S_i)$, and the lemma is proved.
\end{proof}

\begin{corollary}
    Any affine semigroup $S$ is of the form $S_0\cup (X+S)$, where $X\subset S$ is a $S$-incomparable set.
\end{corollary}

\begin{proof}
A consequence of Lemma \ref{lemma_decomposition} and applying Theorem \ref{thrIdealX+S}.
\end{proof}

The next proposition shows a family of MED-semigroups determined from its minimal generating set and its associated sets $S_i$. Note that the fixed conditions imply that $S_0=\{0\}$.

\begin{proposition}\label{MEDtipo2}
 Let $S$ be an affine semigroup minimally generated by $E\sqcup A$ with $E=\{n_1,\ldots, n_t\}$, and such that for any $m\in A$, $m-n_k\in \CaC_S$ for some $n_k\in E$. If the set $S_k$ is a semigroup, then $S$ is a MED-semigroup.
\end{proposition}

\begin{proof}
    Given any $m,n\in A$, we have that $m-n_k, n-n_k\in S_k$. Since $S_k$ is a semigroup, $m-n_k+ n-n_k\in S_k$. Hence, $m+ n-n_k \in S$. By Proposition \ref{caracterizacion_MED}, we conclude $S$ is a MED-semigroup.
\end{proof}

\begin{example}
Let $S$ be again the MED-semigroup considered in Example \ref{MED_example2}. Note that $S$ satisfies the hypothesis of Proposition \ref{MEDtipo2} since $(8, 2)-(5,1), (9, 2)-(5,1), (12, 3)-(5,1)\in \CaC_S$, and $S_1$ is a semigroup (in particular, $S_1$ is equal to $\CaC_S$).
\end{example}

We have seen two distinct families of MED-semigroups. As far as the authors are aware, no additional constructions have been identified. We threw the open problem of whether more distinct constructions exist beyond those already proposed for obtaining MED-semigroups.

\section{Apery sets and the membership problem for affine semigroups}\label{AperySection}

In this section, we show an algorithm for checking whether an integer vector belongs to an affine simplicial semigroup using some of its Apery sets. For some algorithms appearing in this work, we need to be able to compute the set $\cap _{i\in [t]} \Ap(S,m_i)$ for $M=\{m_1,\ldots ,m_t\}\subset S$ such that $m_i\in \tau_i\setminus\{0\}$ for any $i\in [t]$. To compute the intersection $\cap _{i\in [t]} \Ap(S,m_i)$, the following lemma can be used, which has a straightforward proof. 
Alternative methods using Gröbner basis techniques to compute this intersection are discussed in \cite{shortOjedaVigneron} and \cite{shortPison}. For these methods, you have to introduce a new variable for each element in $M$ which is not a minimal generator of $S$.

\begin{lemma}\label{GammaLemma}
Let $S$ be an affine semigroup, and $M=\{m_1,\ldots ,m_t\}\subset S$ such that $m_i\in \tau_i\setminus\{0\}$ for any $i\in [t]$. Then, $\cap _{i\in [t]} \Ap(S,m_i)= \{s\in \Gamma \mid s-m_i\notin S, \forall i\in [t]\}$.
\end{lemma}

\begin{example}
    Let $S$ be again the affine semigroup considered in Example \ref{MED_example2}. By Lemma \ref{GammaLemma}, we have that $\Ap(S,(5,1))\cap \Ap(S,(6,2)) = \{(0, 0), (8, 2), (9, 2), (12, 3)\}$. Hence, as we have just known, $S$ is a MED-semigroup. 
\end{example}

Coming back to the membership problem, note that for any $x\in \CaC_S$ such that $x-n\notin S$ for all $n\in E$, there are only two possibilities: $x\notin S$ or $x\in \cap _{i\in [t]} \Ap(S,n_i)$, i.e. $x$ belongs to $S$. This simple fact lets us introduce Algorithm \ref{pertenencia_con_apery} for checking if an element in $\CaC_S$ belongs to $S$.

\begin{algorithm}[H]
\caption{Checking the membership problem for a simplicial affine semigroup.}\label{pertenencia_con_apery}
\KwIn{A $\CaC$-semigroup $S$ minimally generated by $E\sqcup A$, and $x\in \CaC_S$.}
\KwOut{True if $x\in S$, false in the other case.}
$\mathcal A \leftarrow \cap _{i\in [t]} \Ap(S,n_i) \setminus\{0\}$\\
$y\leftarrow x$\\
$v\leftarrow (0,\dots,0)\in\N^t$\\
\For{$i\in [1,t]$}{
    \While{$y-n_i\in\CaC_S$}{
        \If{$y-n_i\in \mathcal A$}{
            \Return{True}
        }
        $y\leftarrow y-n_i$\\
        $v_i\leftarrow v_i+1$\\
    }
}
$V\leftarrow [0,v_1]\times \dots\times [0,v_{t}]$\\
\If{$x-\sum_{i=1}^t \lambda_i n_i\in \mathcal A$ for some $(\lambda_1,\dots,\lambda_t)\in V$}{\Return{True}}
\Return{False}\\
\end{algorithm}

Note that when Algorithm \ref{pertenencia_con_apery} is applied to MED-semigroups, the set used, $\cap_{i \in [t]} \Ap(S, n_i)$, has the lowest possible cardinality. We illustrate this algorithm on a MED-semigroup.

\begin{example}\label{ejMED}
Let $S$ be again the affine semigroup considered in Example \ref{MED_example2}, and we test if the element $x=(31,8)$ belongs to $S$. The element $(v_1,v_2)$ obtained is equal to $(2,3)$. Since $(9,2)\in \cap _{i\in [t]} \Ap(S,n_i)$, then $(9,2)\in S$.
\end{example}

\subsection*{Funding}
The last author is partially supported by grant PID2022-138906NB-C21 funded by MICIU/AEI/ 10.13039/501100011033 and by ERDF/EU.

Consejería de Universidad, Investigación e Innovación de la Junta de Andalucía project ProyExcel\_00868 and research group FQM343 also partially supported all the authors.

Proyecto de investigación del Plan Propio – UCA 2022-2023 (PR2022-004) partially supported the first and third-named authors.

This publication and research have been partially granted by INDESS (Research University Institute for Sustainable Social Development), Universidad de Cádiz, Spain.

\subsection*{Author information}

J. I. Garc\'{\i}a-Garc\'{\i}a. Departamento de Matem\'aticas/INDESS (Instituto Universitario para el Desarrollo Social Sostenible),
Universidad de C\'adiz, E-11510 Puerto Real (C\'{a}diz, Spain).
E-mail: ignacio.garcia@uca.es.

\noindent
R. Tapia-Ramos. Departamento de Matem\'aticas, Universidad de C\'adiz, E-11406 Jerez de la Frontera (C\'{a}diz, Spain).
E-mail: raquel.tapia@uca.es. 

\noindent
A. Vigneron-Tenorio. Departamento de Matem\'aticas/INDESS (Instituto Universitario para el Desarrollo Social Sostenible), Universidad de C\'adiz, E-11406 Jerez de la Frontera (C\'{a}diz, Spain).
E-mail: alberto.vigneron@uca.es.

\end{document}